\documentclass[a4paper,10pt]{article}
\usepackage{amsmath,amssymb,amsthm,amscd}
\usepackage[mathscr]{eucal}
\usepackage{multirow}
\usepackage{hhline}

\theoremstyle{plain}
\newcommand{\cA}{\mathcal{A}}
\newcommand{\cO}{\mathcal{O}}

\newtheorem{theorem}{Theorem}[section]
\newtheorem{corollary}[theorem]{Corollary}
\newtheorem{lemma}[theorem]{Lemma}
\newtheorem*{claim}{Claim}
\newtheorem{proposition}[theorem]{Proposition}

\theoremstyle{definition}
\newtheorem*{question}{Question}
\newtheorem{definition}[theorem]{Definition}
\newtheorem{remark}[theorem]{Remark}

\newcommand{\e}{\varepsilon}
\newcommand{\aet}{\mathrm{\acute{e}t}}
\newtheorem*{acknowledgements}{Acknowledgements}
\title{Torsion points of abelian varieties
with values in infinite extensions over a $p$-adic field}

\author{Yoshiyasu Ozeki}

\begin{document}
\maketitle
\pagestyle{headings}

%%%%%%%%%%%%%%%%%%%%%%%%%%%%%%%%%%%%%%%%%%%%%%%%%%
%SECTION1%
%%%%%%%%%%%%%%%%%%%%%%%%%%%%%%%%%%%%%%%%%%%%%%%%%%

\begin{abstract}
Let $A$ be an abelian variety
over a $p$-adic field $K$ and
$L$ an algebraic infinite extension over $K$.
We consider the finiteness
of the torsion part of the group of rational points
$A(L)$ under some assumptions.
In 1975, Hideo Imai proved that
such a group
is finite if $A$ has good reduction and
$L$ is the cyclotomic $\mathbb{Z}_p$-extension of $K$.
In this talk, first we show
a generalization of Imai's result in the case where
$A$ has ordinary good reduction.
Next we give some finiteness results
when $A$ is an elliptic curve and $L$ is
the field generated by the $p$-power
torsion of an elliptic curve.
\end{abstract}

\section{Introduction}

Let $K$ be a finite extension field of 
the $p$-adic number field $\mathbb{Q}_p$ with residue field $k$
and fix an algebraic closure $\bar K$ of $K$.
Let $A$ be an abelian variety over $K$.
If $L\subset \bar K$ is a finite extension over $K$,
it is well-known that the torsion part of the $L$-rational points $A(L)$
is finite 
(cf.\ \cite{Ma55}, Thm.\ 7).
On the other hand, in general, 
we do not know whether the torsion part of 
$A(L)$ is finite or infinite
if $L\subset \bar K$ is an infinite algebraic extension over $K$. 
We are interested in understanding whether 
the torsion part of $A(L)$ is finite or infinite.
As one of the known results, 
Imai (\cite{Im75}) proved that the torsion part of $A(K(\mu_{p^{\infty}}))$ 
is finite if $A$ has potential good reduction, 
where $K(\mu_{p^{\infty}})$ is the smallest field 
containing $K$ and all $p$-power roots of unity.
In this connection our first result in this paper is:
%%%%%%%%%%%%%%%%%%%%%%%%%%%%%%%%%%%%%%%%%%%%%%%%
%Main result%
%%%%%%%%%%%%%%%%%%%%%%%%%%%%%%%%%%%%%%%%%%%%%%%%
\begin{theorem}[Thm.\ \ref{Thm:4-1} and Cor.\ \ref{Cor:3-11}]
\label{Thm:1-1}
Let $A$ be an abelian variety over $K$
which has potential ordinary good reduction.
Let $L$ be an algebraic extension of $K$
with residue field $k_L$.

\noindent
$(1)$ Assume that the residue field of $L(\mu_{p^{\infty}})$ is a potential 
prime-to-$p$ extension of $k$ $($in the sense of Def.\ \ref{Def:2-1}$)$.
Then $A(L)[p^{\infty}]$ is finite.

\noindent
$(2)$ If $L$ contains $K(A[p])$ and $K(\mu_{p^{\infty}})$, then
$k_L$ is a potential 
prime-to-$p$ extension of $k$ 
if and only if 
$A(L)[p^{\infty}]$ is finite.

\noindent
$(3)$ Assume that $L(\mu_{p^{\infty}})$ is a Galois extension of $K$  
whose residue field is finite.
Then the torsion part of $A(L)$ is finite. 
\end{theorem} 

Here we denote by $K(A[p])$
the field generated by the coordinates of all $p$-torsion 
points of $A$.
This theorem is a generalization of Imai's theorem 
under the hypothesis that $A$ has ordinary good reduction.
We can also obtain the global case of this theorem, 
see Section 2.4.  
As an easy consequence of Thm.\ \ref{Thm:1-1} (2),
we see that the group $A(K(\mu_{\infty}))[p^{\infty}]$ is infinite for 
an abelian variety $A$ as in the theorem 
if we replace $K$ with its suitable finite extension
(here we denote by $K(\mu_{\infty})$ the field 
obtained by adjoining all roots of unity to $K$). 
We shall point out that such a group must be finite
\textit{in the global case}
by a result of Ribet (\cite{Ri81}). 

Next we consider the finiteness of the torsion part of 
$A(L)$ for $L=K_{B,p}:=K(B[p^{\infty}])$, 
the field generated by the coordinates of all $p$-power torsion 
points of a semiabelian variety $B$.
For example, 
$K_{\mathbb{G}_m,p}$ is  
the cyclotomic field $K(\mu_{p^{\infty}})$, where
 $\mathbb{G}_m$ is the multiplicative group over $K$.
Hence Imai's theorem is a result on the 
torsion part of $A(K_{\mathbb{G}_m,p})$. 
From such a point of view,
we raise the following question:

\noindent
%%%%%%%%%%%%%%%%%%%%%%%%%%%%%%%%%%%%%%%%%%%%%%%%
%question%
%%%%%%%%%%%%%%%%%%%%%%%%%%%%%%%%%%%%%%%%%%%%%%%%
\begin{question}
Let the notation be as above.
Then is the torsion part of 
$A(K_{B,p})$ finite?
\end{question}

The torsion part of $A(K_{B,p})$ is finite if and only if 
$A(K_{B,p})[\ell^{\infty}]$ is finite for all primes $\ell$ and
$A(K_{B,p})[\ell^{\infty}]=0$ for almost all primes $\ell$.
For any prime $\ell\not =p$, it is easy to see that the $\ell$-part of  
$A(K_{B,p})$ is finite 
(cf.\ Prop.\ \ref{Pro:5-1}). 
Hence we are interested in the finiteness of the $p$-part of $A(K_{B,p})$.
If $A=E_1$ and $B=E_2$
are elliptic curves, we can gain various results
by distinguishing the reduction types of $E_1$ and $E_2$,
see the table below.

\begin{theorem}
\label{Thm:1-2}
The finiteness of $E_1(K_{E_2,p})[p^{\infty}]$ is as follows: 

\begin{center}
\begin{tabular}{|c|c|c|c|c|c|} 
\hline
\multicolumn{2}{|c|}{$E_1$} & 
\multicolumn{2}{|c|}{$E_2$} & 
"$E_1(K_{E_2,p})[p^{\infty}]$ & 
$\mathrm{statement}$ 
\\ \hline \hline

\multicolumn{2}{|c|}{\multirow{3}{1.5em}{$\mathrm{ord}$}} & 
\multicolumn{2}{|c|}{$\mathrm{ord}$} & 
`` $\infty$'' ${}^{*_1}$ & $\mathrm{Prop.\ \ref{Thm:5-1}}$\\
\hhline{|~~----|}
\multicolumn{2}{|c|}{} & 
\multicolumn{2}{|c|}{$\mathrm{ss}$} 
& $\mathrm{finite}$ & $\mathrm{Cor.\ \ref{Cor:4-1}}$\\
\hhline{|~~----|}
\multicolumn{2}{|c|}{} & 
\multicolumn{2}{|c|}{$\mathrm{mult}$} & 
$\mathrm{finite}$ & $\mathrm{Cor.\ \ref{Cor:4-1}}$\\ \hline
\hline

\multirow{8}{0.8em}{$\mathrm{ss}$}&
\multirow{4}{2em}{$\mathrm{FCM}$} & 
\multirow{2}{1.5em}{$\mathrm{ord}$} &
$\mathrm{CM}$ & $\mathrm{finite}^{*_2}$ & $\mathrm{Prop.\ \ref{Pro:5-10}}$\\
\hhline{|~~~---|} & & &
$\mathrm{non}$-$\mathrm{CM}$ & $\mathrm{finite}$ & $\mathrm{Prop.\ \ref{Pro:5-10}}$\\
\hhline{|~~----|} & & 
\multirow{2}{0.9em}{$\mathrm{ss}$} &
$\mathrm{FCM}$ & `` $\mathrm{finite}$'' & 
$\mathrm{Prop.\ \ref{Pro:5-6}}$\\ 
\hhline{|~~~---|} & & &
$\mathrm{non}$-$\mathrm{FCM}$ & $\mathrm{finite}$ & $\mathrm{Prop.\ \ref{Pro:5-10}}$\\
\hhline{|~-----|} & 
\multirow{4}{2em}{$\mathrm{non}$-$\mathrm{FCM}$} & 
\multirow{2}{1.5em}{$\mathrm{ord}$} &
$\mathrm{CM}$ & $\mathrm{finite}$ & $\mathrm{Prop.\ \ref{Pro:5-10}}$\\
\hhline{|~~~---|} & & &
$\mathrm{non}$-$\mathrm{CM}$ & $\mathrm{finite}$ & $\mathrm{Prop.\ \ref{Pro:5-10}}$\\
\hhline{|~~----|} & & 
\multirow{2}{0.9em}{$\mathrm{ss}$} &
$\mathrm{FCM}$ & $\mathrm{finite}$ & $\mathrm{Prop.\ \ref{Pro:5-10}}$ \\ 
\hhline{|~~~---|} & & &
$\mathrm{non}$-$\mathrm{FCM}$ & $\mathrm{finite}$ $\mathrm{or}$
$\infty^{*_3}$ & 
$\mathrm{Prop.\ \ref{Pro:5-7}}$\\ \hline
\hline

\multicolumn{2}{|c|}{$\mathrm{split\ mult}$} & 
\multicolumn{2}{|c|}{$\mathrm{any}$}& 
$\infty$ & $\mathrm{Prop.\ \ref{Pro:5-2}}$\\ \hline
\hline

\multicolumn{2}{|c|}{$\mathrm{non}$-$\mathrm{split\ mult}$} & 
\multicolumn{2}{|c|}{$\mathrm{any}$}& 
`` $\mathrm{finite}$'' & $\mathrm{Prop.\ \ref{Pro:5-3}}$\\
\hline
\end{tabular}
\end{center}
\end{theorem}

Here ``ord'', ``ss'', ``mult'', 
``CM'' and ``FCM'' in the above table stand for 
ordinary, supersingular, multiplicative, complex multiplication
and formal complex multiplication, respectively.
The symbols $*_1$, $*_2$ and $*_3$ in the table imply the followings:

\bigskip

$*_1\cdots$  $E_1$, $E_2$ : ordinary good reduction

\qquad \qquad $\Rightarrow$ 
$E_1(K_{E_2,p})[p^{\infty}]$ is infinite in many cases.

$*_2\cdots$ $E_1$ : supersingular good reduction, 
formal complex multiplication
 
\qquad \ \ $E_2$ : ordinary good reduction, complex multiplication

\qquad \qquad $\Rightarrow$
$E_1(K_{E_2,p})[p^{\infty}]$ is finite in all cases.

$*_3\cdots$ $E_1$, $E_2$ : supersingular good reduction, 
formal complex multiplication
 
\qquad \qquad $\Rightarrow$
$E_1(K_{E_2,p})[p^{\infty}]$ may be finite or infinite (case by case).\\

For more precise information of Thm.\ \ref{Thm:1-2}, see the corresponding
statements of the table.

\begin{acknowledgements}
The author would like to express his sincere gratitude to Professor 
Yuichiro Taguchi who proposed the theme of this paper
to him and gave him useful advices, especially
the proof of Prop.\ \ref{Pro:5-7}. 
The author also would like to thank Akio Tamagawa
for pointing out a mistake of Thm.\ \ref{Thm:1-1}.

\end{acknowledgements}

%%%%%%%%%%%%%%%%%%%%%%%%%%%%%%%%%%%%%%%%%%%%%%%
%SECTION2%
%%%%%%%%%%%%%%%%%%%%%%%%%%%%%%%%%%%%%%%%%%%%%%%

\section{Finiteness theorems for abelian varieties}
In this section, we prove Thm.\ \ref{Thm:1-1}   
given in the Introduction and 
consider the global case for this theorem.

Let $p$ be a prime number. 
Let $A$ be an abelian variety over a field $K$ of dimension $d$
and $A^{\vee}$ the dual abelian variety of $A$.
Fix an algebraic closure $\bar K$ of $K$
and a separable closure $K^{\mathrm{sep}}$ of $K$ in $\bar K$.
Put $G_K:=\mathrm{Gal}(K^{\mathrm{sep}}/K)$, 
the absolute Galois group of $K$.
For any algebraic extension $L$ over $K$ 
and any integer $n>0$,
We denote by $A(L)[p^n]$ the kernel of 
the multiplication-by-$p^n$ map of the
$L$-rational points $A(L)$ of $A$
and put $A(L)[p^{\infty}]:=\underset{n>0}{\cup}A(L)[p^n]$.
We denote by $K_{A[p^n]}=K(A[p^n])$ 
the field generated by the coordinates of $A(\bar K)[p^n]$.
Put $K _{A,p}=K(A[p^{\infty}]):=\underset{n>0}{\cup} K(A[p^n])$. 
Then the field $K_{A,p}$ contains $K(\mu_{p^{\infty}})$,
the field adjoining all $p$-power roots of unity to $K$.
The natural continuous 
representation associated to the Tate module $T_p(A)$ of $A$ 
is denoted by 
\[
\rho_{A,p}:G_K\to GL(T_p(A))\simeq GL_h(\mathbb{Z}_p)
\]
for some $h\ge 0$ and denote its residual representation by
\[
\bar \rho_{A,p}:G_K\overset{\rho_{A,p}}{\to} GL_{h}(\mathbb{Z}_p)
\overset{\mathrm{mod}\ p}{\to} GL_{h}(\mathbb{F}_p)
\ (\simeq GL(A(K^{\mathrm{sep}})[p])). 
\]
By abuse of notation, we shall consider 
$\rho_{A,p}$ as the representation of 
$V_p(A)=T_p(A)\otimes_{\mathbb{Z}_p} \mathbb{Q}_p$.
We note that the definition field of the representation 
$\rho_{A,p}$ (resp.\ $\bar \rho_{A,p}$) 
is the field $K_{A,p}$ (resp.\ $K_{A[p]}$).

\subsection{Some properties of torsion points}
In this subsection, we collect some (in-)finiteness
properties of the torsion part of abelian varieties 
which are easy to prove.    

First we note the following proposition, which plays an important role 
throughout this paper:
\begin{proposition}
$A(L)[p^{\infty}]$ is finite
if and only if the fixed subgroup $T_p(A)^{G_L}$ of $T_p(A)$ 
by the absolute Galois group $G_L$ of $L$ 
is $0$.
\end{proposition}
\begin{proof}
This immediately follows from the definition of Tate module 
and the fact that each $A(L)[p^n]$ is finite. 
\end{proof}

The field $K(A[p])$ is a finite Galois extension of $K$
whose Galois group $\mathrm{Gal}(K(A[p])/K)$ is a 
subgroup of $GL_h(\mathbb{F}_p)$ for the integer 
$h\ge 0$ such that $A(\bar K)[p]\simeq (\mathbb{Z}/p\mathbb{Z})^{\oplus h}$.
Put $g_{h,p}:=(p^h-1)(p^h-p)\cdots (p^h-p^{h-1})$,
the order of $GL_h(\mathbb{F}_p)$.

%%%%%%%%%%%%%%%%%%%%%%%%%%%%%%%%%%%%%%%%%%%%%
%prop.%
%%%%%%%%%%%%%%%%%%%%%%%%%%%%%%%%%%%%%%%%%%%%%
\begin{proposition}
$(1)$ If the absolute Galois group $G_K$ of $K$ is 
an inverse limit of finite groups of order prime to $p$,
then   
\[
A(K_{A[p]})[p^{\infty}]=A(K^{\mathrm{sep}})[p^{\infty}].
\]

\noindent
$(2)$ If the absolute Galois group $G_K$ of $K$ is 
an inverse limit of finite groups of order prime to $g_{h,p}$,
then   
\[
A(K)[p^{\infty}]=A(K^{\mathrm{sep}})[p^{\infty}].
\]

\end{proposition}
\begin{proof}
$(1)$ If the group $A(K^{\mathrm{sep}})[p^{\infty}]$ is 0,
there is nothing to prove and hence 
we may assume that $A(K^{\mathrm{sep}})[p^{\infty}]\not =0$. 
Hence, for some $h>0$, $\rho_{A,p}|_{G_{K_{A[p]}}}$ has values 
in the kernel of the reduction map
$GL_h(\mathbb{Z}_p)\to GL_h(\mathbb{F}_p)$,
which is a pro-$p$ group.
Since $G_{K_{A[p]}}$ is an inverse limit of 
finite groups of order prime to $p$,
the representation $\rho_{A,p}|_{G_{K_{A[p]}}}$ is trivial and thus we have 
\[
A(\bar K)[p^{\infty}]=
A(\bar K)[p^{\infty}]^{\mathrm{ker}\rho_{A,p}}\subset
A(\bar K)[p^{\infty}]^{G_{K_{A[p]}}}=
A(K_{A,p})[p^{\infty}].
\]
This completes the proof of $(1)$.

$(2)$ We use the same argument as $(1)$ 
except only that 
we do not need to consider $K_{A[p]}$ since $\rho_{A,p}|_{G_K}$
is already trivial.
\end{proof}

If there is a Galois equivariant homomorphism 
among two Tate modules of abelian varieties,
we can see some infiniteness properties about torsion points
of abelian varieties. 
%%%%%%%%%%%%%%%%%%%%%%%%%%%%%%%%%%%%%%%%%%%%%%%%%%%
%Prop.%
%%%%%%%%%%%%%%%%%%%%%%%%%%%%%%%%%%%%%%%%%%%%%%%%%%%
\begin{proposition}
\label{Pro:2-0}Let $A$ and $B$ be abelian varieties over $K$.
If 
\[
\mathrm{Hom}_{\mathbb{Z}_p[G_K]}(T_p(B),T_p(A))\not =0,
\]
then $A(K_{B,p})[p^{\infty}]$ 
is infinite.
\end{proposition}

\begin{proof}
Take $f$ to be a non-trivial element of 
$\mathrm{Hom}_{\mathbb{Z}_p[G_K]}(T_p(B),T_p(A))$.
Then $T_p(B)/\mathrm{ker}(f)$ is a non-zero 
subspace of $T_p(A)$ with trivial $G_{K_{B,p}}$-action.
This implies the desired statement.  
\end{proof}

\begin{definition}
\label{Def:2-1}
Let $L$ be an algebraic extension of $K$.

\noindent
$(1)$ We say that $L$ is a \textit{prime-to-$p$ extension of $K$}  
if $L$ is a union of finite extensions over $K$ of degree 
prime-to-$p$.

\noindent
$(2)$ We say that $L$ is a \textit{potential
prime-to-$p$ extension of $K$}
if $L$ is a prime-to-$p$ extension over some finite extension field of $K$. 
\end{definition}
\begin{proposition}
\label{Pro:4-0}
Let $A$ be an abelian variety over $K$.
Assume that $A$ has the following property:
for any finite Galois extension $K'$ of $K$, 
the torsion part of $A(K')$ is finite.
Then $A(L)[p^{\infty}]$ is finite for 
any potential prime-to-$p$ Galois extension $L$ of $K$.
\end{proposition}
\begin{proof}
The assertion follows from the facts that
\[
A(L)[p^{\infty}]=A(L\cap K_{A,p})[p^{\infty}]
\]
and $L\cap K_{A,p}$ is a finite Galois extension of $K$.
\end{proof}

Note that we can apply Prop.\ \ref{Pro:4-0} if 
$K$ is one of the following fields:

(i) a finitely generated field over a prime field,

(ii) a finite extension of $\mathbb{Q}_p$.\\ 

\begin{proposition}
\label{Pro:3-10}
Let $A$ be an abelian variety over a finite field $k$ 
of characteristic $p>0$.
Let $k'$ be an algebraic extension of $k$
and $k_p$ the maximal 
pro-$p$-extension of $k$ in $\bar k$.
Assume that $A(\bar k)[p]$ is not trivial.

\noindent
$(1)$ Suppose that $A(\bar k)[p]$ is rational over $k$.
Then $A(k')[p^{\infty}]$ is finite if and only if
$k'$ is a potential prime-to-$p$ extension over $k$.
Furthermore in the other case, $A(k')[p^{\infty}]=A(\bar k)[p^{\infty}]$.

\noindent
$(2)$ $k_{A,p}=k_p(A[p])$.
\end{proposition}

\begin{proof}
(1) First we note that $k_p=k_{A,p}$
because $A(\bar k)[p]\not =0$ and $k$ is a finite field
of characteristic $p$.
If $k'$ is a potential prime-to-$p$ extension over $k$,
Prop.\ \ref{Pro:4-0} implies that $A(k')[p^{\infty}]$ is finite.
In the other case, $k'$ contains $k_p$ and hence 
$A(k')[p^{\infty}]$ is infinite.

\noindent
(2) If we put $k':=k(A[p])$, we see that 
$k_{A,p}$ coincides with $k'_p$, 
the maximal pro-$p$-extension of $k'$.  
Since $k'_p=k_p(A[p])$,
we have done.
\end{proof}

\subsection{Finiteness theorems for 
abelian varieties with ordinary good reduction}
Let $K$ be a finite extension field of $\mathbb{Q}_p$
with integer ring $\cO_K$ and residue field $k$.  
Let $I_K$ be the inertia subgroup of $G_K$.  
If $A$ has good reduction over $K$,
we denote by $\tilde A$ the reduction of $A$ over $k$. 
For any $p$-divisible group $G$ over $\cO_K$,
we denote its Tate module, Tate comodule,
connected component and maximal \'etale quotient by
$T_p(G)$, $\Phi_p(G)$, $G^0$ and $G^{\aet}$, 
respectively and put $V_p(G):=T_p(G)\otimes_{\mathbb{Z}_p} \mathbb{Q}_p$.

Before proving Thm.\ \ref{Thm:1-1} given in the Introduction, 
we shall show the proposition  
related with the matrix of the representation attached to abelian varieties
by an analogous proof
due to Conrad (\cite{Co97}, Thm.\ 1.1).
%%%%%%%%%%%%%%%%%%%%%%%%%%%%%%%%%%%%%%%%
%Proposition 4.2.%
%%%%%%%%%%%%%%%%%%%%%%%%%%%%%%%%%%%%%%%%
\begin{proposition}
\label{Pro:4-2} Let $A$ be an abelian variety over $K$ of dimension $d$
which has good reduction.

\noindent
$(1)$ The representation 
${\rho}_{A,p}$ has the form
\[
\left(\begin{matrix}
  S_A & U_A \\  0 & T_A
  \end{matrix}\right)
\]
with respect to a suitable basis of $T_p(A)$. Here,
for some integer $0\leq f\leq d $,
 
%\noindent
$(\mathrm{i})$ $S_A:G_{K}\to GL_{2d-f}(\mathbb{Z}_{p})$ 
is a continuous homomorphism,

%\noindent
$(\mathrm{ii})$ $T_A:G_{K}\to GL_{f}(\mathbb{Z}_{p})$  
is an unramified continuous homomorphism and

%\noindent 
$(\mathrm{iii})$ $U_A:G_{K}\to M_{2d-f,f}(\mathbb{Z}_{p})$ is a map.

\noindent
$(2)$ If $A$ has ordinary good reduction over $K$, 
then $f=d$ 
and 
$S_A|_{I_{K}}$ is conjugate with the 
direct sum of the $p$-adic cyclotomic characters $\e^{\oplus d}$.

\noindent
$(3)$ If $A$ has ordinary good reduction over $K$,  
the map $S_A$ is conjugate with the map 
$\e\cdot ({^t}T_{A^{\vee}})^{-1}$. 
\end{proposition}

Here, explicitly, the map $\e^{\oplus d}:G_K\to GL_d(\mathbb{Z}_p)$ is 
given by the equation
\[
\e^{\oplus d}(\sigma)=\mathrm{diag}(\e(\sigma))\in 
GL_{d}(\mathbb{Z}_p), 
\]
the diagonal matrix with coefficients $\e(\sigma)$
for all $\sigma\in G_K$,
and for any map $T:G_K\to GL_d(\mathbb{Z}_p)$,
we denote by  ${^t}T$ the map $G_K\to GL_d(\mathbb{Z}_p)$ defined by 
\[
{^t}T(\sigma):={^t}(T(\sigma))\in GL_d(\mathbb{Z}_p) 
\]
for all $\sigma\in G_K$, where ${^t}(T(\sigma))$ is the transposed matrix
of $T(\sigma)$. 
\begin{proof}[Proof of Prop.\ \ref{Pro:4-2}]
(1) Let $\cA$ be the N\'eron model of $A$ over $\cO_K$
and $\cA(p)$ the $p$-divisible group associated to $\cA$.
The connected-\'etale sequence of $\cA(p)$ 
induces the exact sequence
\[
 0\to V_p(\cA(p)^0)\to V_p(\cA(p))\to
 V_p(\cA(p)^{\aet})\to 0
\]
of $\mathbb{Q}_p[G_K]$-modules.
The desired decomposition of $\rho_{A,p}$
can be obtained by this sequence.

\noindent
(2) Assume the reduction type of $\cA$ over $\cO_K$ is ordinary.
Let $\hat {K^{ur}}$ be the completion of 
the maximal unramified extension $K^{ur}$ over $K$ 
and $\cO_{\hat {K^{ur}}}$ the integer ring of $\hat {K^{ur}}$.
Then $\widehat {\cA}_{\cO_{\hat {K^{ur}}}}:=
\widehat {\cA}\times_{\cO_K} \mathrm{Spf}({\cO_{\hat {K^{ur}}}})$
is isomorphic to 
$\widehat {\mathbb{G}}_{m}^{\oplus d}$ over $\cO_{\hat {K^{ur}}}$,
where $\widehat \cA$ is the formal completion of $\cA$ along its zero section
and $\widehat {\mathbb{G}}_m$ is the formal multiplicative group over 
$\cO_{\hat {K^{ur}}}$ (cf.\ \cite{Ma72}, Lem.\ 4.26 and Lem.\ 4.27).
This implies the assertion $(2)$.

\noindent
(3) Consider the following two exact sequences as $G_K$-modules;
\begin{align*}
&0\to V_p(\cA(p)^{\aet})^{\vee}\to 
V_p(\cA(p))^{\vee}\to V_p(\cA(p)^0)^{\vee}\to 0,\\
&0\to V_p(\cA^{\vee}(p)^0)(-1)\to V_p(A^{\vee})(-1)
\to V_p(\cA^{\vee}(p)^{\aet})(-1)\to 0, 
\end{align*}
where 
$M^{\vee}=\mathrm{Hom}_{\mathbb{Q}_p}(M,\mathbb{Q}_p)$ is 
the dual of a $\mathbb{Q}_p[G_K]$-module $M$ and 
$M(i)$ is the $i$-th Tate twist of $M$. 
Note that $V_p(\cA(p))^{\vee}\simeq V_p(A^{\vee})(-1)$ as $G_K$-modules.
By taking the functor $\mathrm{H}^0(I_K,-)$ of
the above exact sequences and 
using the assertion $(2)$,
we can see that 
\[
V_p(\cA(p)^{\aet})^{\vee}\simeq 
(V_p(\cA(p))^{\vee})^{I_K}\simeq (V_p(A^{\vee})(-1))^{I_K}\simeq
V_p(\cA^{\vee}(p)^0)(-1)
\]
as $G_K$-modules,
since $I_K$ acts on $V_p(\cA(p)^{\aet})^{\vee}$ and 
$V_p(\cA^{\vee}(p)^0)(-1)$ trivially, and  
$I_K$ acts on $V_p(\cA(p)^0)^{\vee}$ and $V_p(\cA^{\vee}(p)^{\aet})(-1)$ 
by $(\e^{\oplus d})^{-1}$.
We know that the group $G_K$ acts on $V_p(\cA(p)^{\aet})^{\vee}(1)$
by $\e\cdot ({^t}T_A)^{-1}$ and also 
acts on $V_p(\cA^{\vee}(p)^0)$ by $S_{A^{\vee}}$.
Therefore, we see that  
$\e\cdot ({^t}T_A)^{-1}$ is conjugate with $S_{A^{\vee}}$ and
thus we finish the proof 
of the assertion $(3)$ after replacing $A$ with $A^{\vee}$.  
\end{proof}

%%%%%%%%%%%%%%%%%%%%%%%%%%%%%%%%%%%
%remark4.3.%
%%%%%%%%%%%%%%%%%%%%%%%%%%%%%%%%%%%
\begin{remark}
\label{Rem:4-1}

\noindent
(1) The above proof implies the following:
the integer $f$ is equal to the dimension of $V_p(\tilde A)$
and 
the map $S_A$ is the natural continuous homomorphism
\[
G_K\to GL(T_p({\cA(p)^0}))\simeq GL_{2d-f}(\mathbb{Z}_p)  
\]
and the map $T_A$ is the natural continuous homomorphism
\[
G_K\to GL(T_p(\tilde A))\simeq GL_{f}(\mathbb{Z}_p).
\]

\noindent
(2) Suppose that $A$ has ordinary good reduction over $K$.
We denote the semi-simplification of $\rho_{A,p}$ by $\rho_{A,p}^{ss}$.
The criterion of N\'eron-Ogg-Shafarevich implies that
the representation $\rho_{A,p}^{ss}|_{G_L}$ 
factors through $G_{k_L}$ for any algebraic extension $L$ of $K$ such that 
$L$ contains all $p$-power roots of unity. 
\end{remark}
Now we can show our main finiteness theorem
which is given in the Introduction. 

%%%%%%%%%%%%%%%%%%%%%%%%%%%%%%%%%%%
%thm. 4.4.%
%%%%%%%%%%%%%%%%%%%%%%%%%%%%%%%%%%% 
\begin{theorem}
\label{Thm:4-1} 
Let $A$ be an abelian variety over $K$
which has potential ordinary good reduction.
Let $L$ be an algebraic extension of $K$.

\noindent
$(1)$ Assume that the residue field of $L(\mu_{p^{\infty}})$ is a potential 
prime-to-$p$ extension of $k$.
Then $A(L)[p^{\infty}]$ is finite.

\noindent
$(2)$ Assume that $L(\mu_{p^{\infty}})$ is a Galois extension of $K$  
whose residue field is finite.
Then the torsion part of $A(L)$ is finite. 
 
\end{theorem}

\begin{proof}
$(1)$ 
By extending $K$ and $L$,
we may assume that $A$ has ordinary good reduction over $K$
and $L$ contains $K(\mu_{p^{\infty}})$. 
Put $d:=\mathrm{dim}(A)$ and 
denote by $k_L$ the residue field of $L$.
We have the following form of 
the representation $\rho_{A,p}$ for a suitable 
basis of the $\mathbb{Z}_{p}$-module $T_p(A)$:
\[
\left(\begin{matrix}
  S_A & U_A \\  0 & T_A
  \end{matrix}\right),
\]
with certain
$S_A=\e\cdot ({^t}T_{A^{\vee}})^{-1}
:G_{K}\to GL_{d}(\mathbb{Z}_{p})$, $T_A:G_{K}\to GL_{d}(\mathbb{Z}_{p})$ 
and $U_A:G_{K}\to M_d(\mathbb{Z}_{p})$ as in Prop.\ \ref{Pro:4-2}.
In this proof,
we shall fix the above basis and 
identify $V_p(A)$ with $\mathbb{Q}_p^{\oplus 2d}=M_{2d,1}(\mathbb{Q}_p)$.
Since $L$ contains all $p$-power roots of unity,
we know that the maps $S_A|_{G_L}$ and $T_A|_{G_L}$ 
factor through the absolute Galois group 
$G_{k_L}:=\mathrm{Gal}(k^{\mathrm{sep}}/k_L)$ of $k_L$.
Thus the maps $S_A|_{G_L}$ and $T_A|_{G_L}$ are determined
by a topological generator $\sigma_L$ of $G_{k_L}$. 
To prove $V_p(A)^{G_L}=0$,
it suffices to show that the matrices $S_A(\sigma_L)$ and 
$T_A(\sigma_L)$ do not have eigenvalue $1$.
First we assume that this assertion about $T_A(\sigma_L)$ is false.  
Then the Tate module $V_p(\tilde A)^{G_L}$ is not zero

and hence $\tilde A(k_L)[p^{\infty}]$ is infinite.
By Prop.\ \ref{Pro:4-0}, 
this contradicts the assumption that the field $k_L$ is 
potential prime-to-$p$ extension of $k$.
Next we show that $S_A(\sigma_L)$ does not have eigenvalue $1$.
Since $L$ contains $K(\mu_{p^{\infty}})$ and 
$A$ has ordinary good reduction,
the map $S_A|_{G_L}$ coincides with
$({^t}T_{A^{\vee}}|_{G_L})^{-1}$.
Hence if we assume $S_A(\sigma_L)$ has eigenvalue $1$,
then $T_{A^{\vee}}(\sigma_L)$ also has eigenvalue $1$.
But this induces a contradiction by the same argument as the above.

\noindent
$(2)$ 
The criterion of N\'eron-Ogg-Shafarevich implies that  
\[
A(L)'\subset A(L(\mu_{p^{\infty}}))'=
A(L(\mu_{p^{\infty}})\cap K^{\mathrm{ur}})'
\]
where the symbol $'$ means the prime-to-$p$ part. 
Hence 
$A(L)'$ is finite because $L(\mu_{p^{\infty}})$ has a finite 
residue field.
Therefore, the assertion $(1)$
implies that the torsion points of $A(L)$ is finite. 
\end{proof}

%%%%%%%%%%%%%%%%%%%%%%%%%%%%%%%%%%%%%%%%%%%
%corollary%
%%%%%%%%%%%%%%%%%%%%%%%%%%%%%%%%%%%%%%%%%%%

\begin{corollary}
\label{Cor:3-11}
Let $A$ be an abelian variety over $K$
which has ordinary good reduction.
Let $L$ be a Galois extension of $K$ with residue field $k_L$.
Assume that $L$ contains $K(\mu_{p^{\infty}})$ and $K(A[p])$.
Then the followings are equivalent $:$

\noindent
$(1)$ $A(L)[p^{\infty}]$ is finite,

\noindent
$(2)$ $A^{\vee}(L)[p^{\infty}]$ is finite,

\noindent
$(3)$ $\tilde A(k_L)[p^{\infty}]$ is finite,

\noindent
$(4)$ $\tilde {A^{\vee}}(k_L)[p^{\infty}]$ is finite,

\noindent
$(5)$ $k_L$ is a potential prime-to-$p$ extension over $k$.
\end{corollary}

Note that the Weil-paring implies that
$L$ contains $K(\mu_{p^{\infty}})$ and $K(A[p])$
if and only if 
$L$ contains $K(\mu_{p^{\infty}})$ and $K(A^{\vee}[p])$
 
\begin{proof}
The equivalence of assertions $(3)$, $(4)$ and $(5)$
follows from Prop.\ \ref{Pro:3-10}, thus 
it is enough to show that $(1)$ is equivalent to 
$(3)$, $(4)$ and $(5)$. 
By Thm.\ \ref{Thm:4-1} (1), the condition $(5)$ implies
the condition $(1)$.

Let us assume that the condition $(4)$ is not satisfied.
Then we know that 
$T_{A^{\vee}}(\sigma_L)$ has eigenvalue $1$,
where $T_{A^{\vee}}$ is a natural unramified homomorphism
$G_K\to GL(T_p(\tilde {A^{\vee}}))$ and 
$\sigma_L$ is a topological generator of $G_{k_L}$.
Hence $\e\cdot ({^t}T_{A^{\vee}})^{-1}(\sigma_L)$ 
has also eigenvalue $1$.
By using this fact and Prop.\ \ref{Pro:4-2}, we can see that 
$T_p(A)^{G_L}$ is not trivial. Therefore, 
the group $A(L)[p^{\infty}]$ is infinite.  
\end{proof}

\subsection{$K^{\mathrm{ur}}$-rational points}
We continue to use the same notation $p$, $K$, $k$, $A$ and $d$ 
as in the previous subsection.   
In this subsection, 
we consider some relations of 
$\rho_{A,p}:G_K\to GL(V_p(A))\simeq GL_{2d}(\mathbb{Q}_p)$ 
with $A(K^{\mathrm{ur}})[p^{\infty}]$.
Our goal in this subsection is to prove that 
\begin{theorem}
\label{Thm:3-12}
Let $E$ be an elliptic curve over $K$ which has ordinary
good reduction. Then the followings are equivalent;

\noindent
$(1)$ $\rho_{E,p}$ is abelian,

\noindent
$(2)$ $\rho_{E,p}|_{I_K}$ is abelian,

\noindent
$(3)$ $E(K^{\mathrm{ur}})[p^{\infty}]$ is infinite.
\end{theorem}

Note that it is known that the condition (1) is equivalent to the
condition 

%\vspace{0.5}
\noindent
$(4)$ $E$ has complex multiplication over $K$.

\bigskip

See \cite{Se89}, A.2.4 for more information.

Now we start with an argument by proving the fact below;

\begin{proposition}
\label{Prop:3-30}
Let $A$ be an abelian variety over $K$
which has ordinary good reduction.
If $\rho_{A,p}|_{I_K}$ is abelian, then 
$A(K^{\mathrm{ur}})[p^{\infty}]$ is infinite.  
\end{proposition}

\begin{proof}
By Prop.\ \ref{Pro:4-2}, we have the following form of 
the representation $\rho_{A,p}|_{I_K}$ for a suitable 
basis of the $\mathbb{Q}_{p}$-module $V_p(A)$:
\[
\left(\begin{matrix}
  \e^{\oplus d}& U \\  0 & I_d
  \end{matrix}\right),
\]
where 
$U$ is a map $I_K\to M_d(\mathbb{Q}_{p})$
and $I_d$ is the unit matrix of $d\times d$.
Since $\rho_{A,p}|_{I_K}$ is abelian,
we see that $U=(\e^{\oplus d}-I_d)U_0$ 
for some $U_0\in M_d(\mathbb{Q}_p)$. 
Then, on $I_K$, 
\[ 
\left(\begin{matrix}
  \e^{\oplus d} & U \\  0 & I_d
  \end{matrix}\right) 
\left(\begin{matrix}
  U_01_d \\ 1_d
\end{matrix}\right)
=
\left(\begin{matrix}
  U_01_d \\ 1_d
\end{matrix}\right),
\]
where $1_d={^t}(1,1,\cdots ,1)\in M_{d,1}(\mathbb{Q}_p)$.
Consequently, we know that $V_p(A)^{I_K}\not= 0$ and 
this implies the desired result.
\end{proof}

In the rest of this subsection we always 
assume that $A$ has ordinary good reduction over $K$.
Let $K^{\mathrm{ab}}$ be the maximal abelian extension of $K$ in $\bar K$.
The exact sequence
\[
0\to \mathrm{ker}(r)\to V_p(A)\overset{r}{\to} V_p(\tilde A)\to 0  
\]
gives the exact sequence
\[
0\to \mathrm{ker}(r)\to 
V_p(A)^{G_{K^{\mathrm{ab}}}}\overset{r}{\to} V_p(\tilde A) \qquad (*)
\]
of $\mathbb{Q}_p[\mathrm{Gal}(K^{\mathrm{ab}}/K)]$-modules.
Here we remark that $G_{K^{\mathrm{ab}}}$ acts trivially
on $\mathrm{ker}(r)$ and $V_p(\tilde A)$.
We define the representation $\rho_{A,p}^{\mathrm{ur}}$ by
the natural action of 
$\mathrm{Gal}(K^{\mathrm{ab}}/K^{\mathrm{ur}})$ on 
$V_p(A)^{G_{K^{\mathrm{ab}}}}=V_p(A(K^{\mathrm{ab}}))$;
\[
\rho_{A,p}^{\mathrm{ur}}:\mathrm{Gal}(K^{\mathrm{ab}}/K^{\mathrm{ur}})\to 
GL(V_p(A)^{G_{K^{\mathrm{ab}}}})=GL(V_p(A(K^{\mathrm{ab}}))).
\]
Now we define the integer $e(A)$ to be  
$e(A):=\mathrm{dim}_{\mathbb{Q}_p}
V_p(A(K^{\mathrm{ab}}))-\mathrm{dim}_{\mathbb{Q}_p}\mathrm{ker}(r)
=\mathrm{dim}_{\mathbb{Q}_p}
V_p(A(K^{\mathrm{ab}}))-d$. 
Then $\rho_{A,p}^{\mathrm{ur}}$
is a $p$-adic representation of dimension $d$ + $e(A)$.
Clearly $0\le e(A)\le d$.
Furthermore, the above sequence $(*)$ implies that 
$\rho_{A,p}^{\mathrm{ur}}$ has the following shape for a suitable 
basis of the $\mathbb{Q}_{p}$-module $V_p(A(K^{\mathrm{ab}}))$:
\[
\rho_{A,p}^{\mathrm{ur}}\simeq
\left(\begin{matrix}
  \e^{\oplus d}& U \\  0 & I_{e(A)}
  \end{matrix}\right),
\]
where 
$U$ is a map $\mathrm{Gal}(K^{\mathrm{ab}}/K^{\mathrm{ur}})
\to M_d(\mathbb{Q}_{p})$
and $I_{e(A)}$ is the unit matrix of $e(A)\times e(A)$.\\ 

%\vspace{-0.5zw}

\noindent
\underline{(I) The case $e(A)\not =0$.}

In this case, we see that $U=(\e^{\oplus d}-I_d)U_0$ 
for some $U_0\in M_{d,e(A)}(\mathbb{Q}_p)$ since 
$\rho_{A,p}^{\mathrm{ur}}$ is abelian. 
Then, on $\mathrm{Gal}(K^{\mathrm{ab}}/K^{\mathrm{ur}})$,
\[ 
\left(\begin{matrix}
  \e^{\oplus d} & U \\  0 & I_{e(A)}
  \end{matrix}\right) 
\left(\begin{matrix}
  U_01_{e(A)} \\ 1_{e(A)}
\end{matrix}\right)
=
\left(\begin{matrix}
  U_01_{e(A)} \\ 1_{e(A)}
\end{matrix}\right),
\]
where $1_{e(A)}={^t}(1,1,\cdots ,1)\in M_{e(A),1}(\mathbb{Q}_p)$.
This equation shows the fact that 
$V_p(A(K^{\mathrm{ab}}))
^{\mathrm{Gal}(K^{\mathrm{ab}}/K^{\mathrm{ur}})}\not =0$, that is, 
$A(K^{\mathrm{ur}})[p^{\infty}]$ is infinite.\\

%\vspace{-0.5zw}

\noindent
\underline{(II) The case $e(A)=0$.}

In this case, we have $\rho_{A,p}^{\mathrm{ur}}=\e^{\oplus d}$ 
and hence $V_p(A(K^{\mathrm{ab}}))
^{\mathrm{Gal}(K^{\mathrm{ab}}/K^{\mathrm{ur}})}=0$, that is, 
$A(K^{\mathrm{ur}})[p^{\infty}]$ is finite.\\

%\vspace{-0.5zw}

Consequently we obtained the following lemma:

\begin{lemma}
\label{Lem:3-13}
Let the notation be as above.
Then $e(A)=0$ if and only if $A(K^{\mathrm{ur}})[p^{\infty}]$ is finite. 
\end{lemma}

Now we can finish the proof of Thm.\ \ref{Thm:3-12}.

\begin{proof}[Proof of Thm.\ \ref{Thm:3-12}]
Since $E$ is an elliptic curve, 
by combining Prop.\ \ref{Prop:3-30} with Lem.\ \ref{Lem:3-13},
we can show the 
desired statement by the following way:
$e(E)\not =0
\Leftrightarrow
e(E)=1
\Leftrightarrow
\mathrm{dim}_{\mathbb{Q}_p}V_p(E(K^{\mathrm{ab}}))=2
\Leftrightarrow
V_p(E(K^{\mathrm{ab}}))=V_p(E)
\Leftrightarrow
G_{K^{\mathrm{ab}}}\subset G_{K_{E,p}}
\Leftrightarrow
\rho_{E,p}$ is abelian. 
\end{proof}

\subsection{Global cases}
Consider ``global cases'' of Thm.\ \ref{Thm:4-1}.
Let $A$ be an abelian variety over a number field $K$. 
It is well-known that
the group $A(L)$ is a finitely generated commutative group
for a finite extension field $L$ of $K$ by 
the theorem of Mordell-Weil-N\'eron-Lang.
In particular its torsion subgroup is finite. 
In the case where $L$ is any algebraic extension of $K$,
there are many results on
the finiteness of torsion points 
of $A(L)$.

Let $v$ be a finite place of $K$.
For any finite extension $K'$ of $K$ and 
any finite place $v'$ of $K'$ above $v$,
we denote the completion of $K'$ at $v'$ by $K'_{v'}$.
More generally, for any algebraic extension $L$ and 
any place $w$ above $v$, we denote
\[
 L_w:=\underset{K'}{\cup} K'_{v'},
\]
where $K'$ runs through all the finite extensions of $K$ in $L$
and $v'$ is the unique place of $K'$ under $w$.
Note that the residue field $k_{L_w}$ of $L_w$ is 
$\underset{K'}{\cup}k_{K_{v'}'}$.

As corollaries of Thm.\ \ref{Thm:4-1},
we can see the ``global cases'' below immediately.

\begin{corollary}
\label{Cor:4-3}Let $K,L,A$ be as above.
Assume that there exist places $v$ of $K$ above $p$ and 
$w_{\infty}$ of $L(\mu_{p^{\infty}})$ above $v$ satisfying 
the following properties:

\noindent
$(\mathrm{i})$ The residue field $k_{w_{\infty}}$ of $L(\mu_{p^{\infty}})$ 
at $w_{\infty}$ is 
a potential prime-to-$p$ extension of the 
residue field $k_v$ of $K$ at $v$.

\noindent
$(\mathrm{ii})$ $A$ has  
potential ordinary good reduction at $v$.

Then $A(L)[p^{\infty}]$ is finite.
\end{corollary}

\begin{corollary}
\label{Cor:4-4}Let $K,L,A$ be as above.
Assume that $L(\mu_{p^{\infty}})$ is a Galois extension of $K$, and 
there exist places $v$ of $K$ above $p$ and 
$w_{\infty}$ of $L(\mu_{p^{\infty}})$ above $v$ satisfying 
the following properties:

\noindent
$(\mathrm{i})$ The residue field $k_{w_{\infty}}$ of $L(\mu_{p^{\infty}})$ 
at $w_{\infty}$ is 
finite.

\noindent
$(\mathrm{ii})$ $A$ has  
potential ordinary good reduction at $v$.

Then the torsion part of $A(L)$ is finite.
\end{corollary}

If we always assume that 
$L$ contains all $p$-power roots of unity,
these corollaries are generalizations
of a result of Greenberg. 
See \cite{Gr03},\ Prop.\ 1.2 (ii).

\section{Finiteness of torsion points for elliptic curves}
We use the same notations as defined at the beginning of the previous
section
(soon we will suppose $K$ to be a finite extension of $\mathbb{Q}_p$).
In particular $A$ is an abelian variety over a field $K$. 
In this section, 
we give results on the Question which is proposed in the Introduction.
At the beginning,
we shall remark the following  proposition related with 
``a torsion problem of two abelian varieties''.
%%%%%%%%%%%%%%%%%%%%%%%%%%%%%%%%%%%%%%%
%proposition%
%%%%%%%%%%%%%%%%%%%%%%%%%%%%%%%%%%%%%%%
\begin{proposition}
\label{Pro:5-1} Let $A$ and $B$ be abelian varieties over a field $K$.
Assume that $K$ has the following property:
the torsion part of $A(K')$ is finite for any finite Galois 
extension $K'$ of $K$.
Then, for two different prime numbers $\ell_1$ and $\ell_2$,
the group $A(K_{B,\ell_2})[\ell_1^{\infty}]$ is finite. 
\end{proposition}
\begin{proof}
This follows immediately from Prop.\ \ref{Pro:4-0}.
\end{proof}

In view of the above proposition,
we will be interested in the finiteness of 
$A(K_{B,\ell_2})[\ell_1^{\infty}]$ with $\ell_1=\ell_2$.

From now on, throughout this Section, we always denote by 
$K$ a finite extension of $\mathbb{Q}_p$.
Some of results on such the (in-)finiteness properties
can be checked immediately, 
by using the results given in the previous section,
as follows.

\begin{proposition}
\label{Cor:4-1} 
Let $E$ be elliptic curves over $K$ which has 
potential multiplicative reduction or 
potential supersingular good reduction.
We assume that $A$ has potential ordinary good reduction over $K$.
Then the torsion part of 
$A(K_{E,p})$ is finite. 
\end{proposition}

\begin{proof}
The two Lie groups $\rho_{E,p}(G_K)$ and $\rho_{E,p}(I_K)$
have the same Lie algebras (cf.\ \cite{Se72}) and hence
the residue field of $K_{E,p}$ 
is a finite field. Therefore, 
the residue field of $K _{E,p}$
is also finite.
Consequently we finish the proof by Thm.\ \ref{Thm:4-1} (2).  
\end{proof}

\begin{proposition}
\label{Pro:4-20}
Let $A$ and $B$ be abelian varieties over $K$ 
which have ordinary good reductions.
Assume that $A(\bar K)[p]$ is rational over $K$.
Then $A(K_{B,p})$ is infinite.
\end{proposition}

\begin{proof}
Since $B(K_{B,p})[p^{\infty}]$ is infinite, 
the residue field $k_{B,p}$ of $K_{B,p}$ is not  
a potential prime-to-$p$ extension of $k$ by Thm.\ \ref{Thm:4-1}. 
Because $A(\bar K)[p]$ is rational over $K$,
Cor.\ \ref{Cor:3-11} shows that $A(K_{B,p})[p^{\infty}]$ is infinite.
\end{proof}

In the rest of this section, 
we discuss the following question:

\begin{question}
Let $A$ and $B$ be abelian varieties over $K$.
When is $A(K_{B,p})[p^{\infty}]$ finite?
\end{question}

Now we are interested in the case where 
$A$ (and $B$) are elliptic curves.
Let us consider the above Question by 
distinguishing the reduction type of $A$.

\subsection{Ordinary good reduction case}
Let $A=E_1$ and $B=E_2$ be two elliptic curves over $K$.
We have already proved in Cor.\ \ref{Cor:4-1} that 
the torsion part of $E_1(K_{E_2,p})[p^{\infty}]$ is finite
if $E_1$ has ordinary good reduction over $K$ and 
$E_2$ has supersingular good reduction or multiplicative reduction over $K$.
We consider the infiniteness 
of $E_1(K_{E_2,p})[p^{\infty}]$ under the condition that 
$E_1$ and $E_2$ have ordinary good reduction over $K$. 
One of the result for the 
finiteness of $E_1(K_{E_2,p})[p^{\infty}]$ has 
given in Prop.\ \ref{Pro:4-20}. However, 
in this $1$-dimensional case,
we will show more precise criterion in 
Thm.\ \ref{Thm:5-1} and Cor.\ \ref{Cor:5-10}.

We shall give some notation that we need.
In the rest of this subsection 
we always assume that 
$E_1$ and $E_2$ have ordinary good reduction over $K$.
Let $\tilde E_1$ and $\tilde E_2$ be the reduction of 
$E_1$ and $E_2$ over $k$, respectively.
For each elliptic curves $E_i$, Put
\[
\chi_i=\rho_{\tilde E_i,p}:G_K\to GL(T_p(\tilde E_i))\simeq 
GL_{1}(\mathbb{Z}_p)=\mathbb{Z}_p^{\times},
\]
and 
\[
\bar \chi_i=\bar \rho_{\tilde E_i,p}:G_K\to 
GL(\tilde E_i(k^{\mathrm{sep}})[p])\simeq 
GL_{1}(\mathbb{F}_p)=\mathbb{F}_p^{\times}. 
\]
Clearly $\chi_i$ and $\bar \chi_i$ are unramified characters by
their definitions. 
It can be checked that each $\chi_i$ is of infinite order.
We know that  
each ${\rho}_{E_i,p}$ has the form
\[
\left(\begin{matrix}
  \e\chi_i^{-1} & u_i \\  0 & \chi_i
  \end{matrix}\right)
\]
with respect to a suitable basis of $T_p(E_i)$.
We fix such a basis and identify 
$T_p(E_i)$ with $\mathbb{Z}_p^{\oplus 2}$. 
%%%%%%%%%%%%%%%%%%%%%%%%%%%%%%%%%%%%%%%%%%%%%
%lemma%
%%%%%%%%%%%%%%%%%%%%%%%%%%%%%%%%%%%%%%%%%%%%%

\begin{proposition}
\label{Thm:5-1}
Let the notations be as above and $p\ge 3$. 
Consider the following four conditions $:$

\noindent
$(a)$ The group $E_1(K_{E_2,p})[p^{\infty}]$ is infinite.

\noindent
$(b)$  $G_{K_{E_2,p}}\subset \mathrm{ker}({\chi_1})$. 

\noindent
$(c)$ $\mathrm{ker}({\chi_2})\subset \mathrm{ker}({\chi_1})$.

\noindent
$(d)$ $\mathrm{Im}(\bar \chi_1)\subset \mathrm{Im}(\bar \chi_2)$.

Then there is the following relation:
$(a)\Leftrightarrow
(b)\Leftarrow
(c)\Leftrightarrow
(d)$.
\end{proposition}

\begin{proof}
If we assume that the condition $(a)$ is satisfied,
there is a $1$-dimensional $G_{K_{E_2,p}}$-invariant 
subspace $W$ of $V_p(E_1)=M_{2,1}(\mathbb{Q}_p)$.
Take any non-zero element 
$\left(\begin{smallmatrix}
  x \\ y
\end{smallmatrix}\right)$  
in $W$.
Since $x$ or $y$ is a non-zero element which is  
invariant under the multiplication by $\chi_1(\sigma)$ for all 
$\sigma\in G_{K_{E_2,p}}$,
we see that 
$V_p(\tilde E_1)^{G_{K_{E_2,p}}}$ has 
non-trivial subspace and hence 
$V_p(\tilde E_1)^{G_{K_{E_2,p}}}=V_p(\tilde E_1)$.
This implies the condition $(b)$.
Conversely we assume that the condition $(b)$ is satisfied.
Then the condition $(a)$ follows from the fact that 
\[
\rho_{E_1,p}|_{G_{K_{E_2,p}}}=
\left(\begin{matrix}
  \e\chi_1^{-1} & u_1 \\  0 & \chi_1
  \end{matrix}\right)
=
\left(\begin{matrix}
  1 & u_1 \\  0 & 1
  \end{matrix}\right). 
\]
\noindent
Because $G_{K_{E_2,p}}\subset \mathrm{ker}(\chi_2)$,
the condition $(c)$ implies $(b)$.
Finally let us show that conditions $(c)$ and $(d)$
are equivalent.
Since $\chi_1$ and $\chi_2$ are unramified,
we may consider each character $\chi_i$ as a 
character 
\[
G_k\overset{\chi_i}{\to} \mathbb{Z}_p^{\times}\simeq 
\mathbb{Z}/(p-1)\mathbb{Z}\times \mathbb{Z}_p,
\]
where $G_k$ is an absolute Galois group of $k$.
By 
$\mathrm{pr}_1:
\mathbb{Z}/(p-1)\mathbb{Z}\times \mathbb{Z}_p\to \mathbb{Z}/(p-1)\mathbb{Z}$
and 
$\mathrm{pr}_2:\mathbb{Z}/(p-1)\mathbb{Z}\times \mathbb{Z}_p\to \mathbb{Z}_p$, 
we denote natural projections.
Let $\sigma_k$ be the Frobenius automorphism of $G_k$ and 
we decompose $\chi_i(\sigma_k)=(m_i,n_i)\in 
\mathbb{Z}/(p-1)\mathbb{Z}\times \mathbb{Z}_p$.
Note that 
$\mathrm{Im}(\chi_i)$ has an infinite image and thus $n_i\not =0$,
since $k$ is a finite field. 
If we identify $G_k$ with $\underset{\ell}{\prod}\mathbb{Z}_{\ell}$,
we see that  
\begin{align*}
\mathrm{ker}(\chi_i)&=
\mathrm{ker}(\mathrm{pr}_1\circ \chi_i)\cap 
\mathrm{ker}(\mathrm{pr}_2\circ \chi_i)\\
&=(p-1)/(\mathrm{gcd}(p-1,m_i)){\underset{\ell}{\prod}\mathbb{Z}_{\ell}}
\cap \underset{\ell\not =p}{\prod}\mathbb{Z}_{\ell}\\
&=(p-1)/(\mathrm{gcd}(p-1,m_i))\underset{\ell\not =p}{\prod}\mathbb{Z}_{\ell}.
\end{align*}

Therefore,
it can be checked that 
the condition $(c)$ 
is equivalent to the condition that
$m_1\cdot \mathbb{Z}/(p-1)\mathbb{Z}\subset 
m_2\cdot \mathbb{Z}/(p-1)\mathbb{Z}$, 
which is $(d)$.
\end{proof}

Let $k_{\tilde E_i[p]}$ be the smallest field extension of $k$ over which
the elements of $\tilde E_i(k^{\mathrm{sep}})[p]$ is rational. 
%%%%%%%%%%%%%%%%%%%%%%%%%%%%%%%%%%%%%%%%%%%%%%%%%%%%%%%%%
%corollary%
%%%%%%%%%%%%%%%%%%%%%%%%%%%%%%%%%%%%%%%%%%%%%%%%%%%%%%%%% 
\begin{corollary}
\label{Cor:5-10}
Let $E_1$ and $E_2$ be elliptic curves over $K$ which have 
ordinary good reduction.
Assume that $p\ge 3$.

\noindent
$(1)$ If $k_{\tilde E_1[p]}\subset k_{\tilde E_2[p]}$,
then $E_1(K_{E_2,p})[p^{\infty}]$ is infinite.

\noindent
$(2)$ If the map $\bar \chi_1$ is trivial,
then 
$E_1(K_{E_2,p})[p^{\infty}]$ is infinite. 

\noindent
$(3)$ If the map $\bar \chi_2$ is
surjective, then
$E_1(K_{E_2,p})[p^{\infty}]$ is infinite.
\end{corollary}
\begin{proof}
All statements will immediately follow from 
Prop.\ \ref{Thm:5-1}.
\end{proof}

\subsection{Supersingular good reduction case}

Let $A=E_1$ and $B=E_2$ be elliptic curves over $K$.
In this subsection we consider the case where 
$E_1$ has (potential) supersingular good reduction over $K$.
We recall the structure of the Lie algebra 
associated to an elliptic curve.

\begin{proposition} [\cite{Se89}, Appendix of Chap.\ 4]
\label{Pro:5-4}
Let $E$ be an elliptic curve over $K$.
Put $\mathfrak{g}:=\mathrm{Lie}(\rho_{E,p}(G_K))$ and 
$\mathfrak{i}:=\mathrm{Lie}(\rho_{E,p}(I_K))$
$($these are Lie subalgebras of $\mathrm{End}_{\mathbb{Q}_p}(V_p(E))$$)$.

\noindent
$(1)$ If $E$ has ordinary good reduction with 
complex multiplication, then 
$\mathfrak{g}$ is a split Cartan subalgebra of 
$\mathrm{End}_{\mathbb{Q}_p}(V_p(E))$ and 
$\mathfrak{i}$ is a $1$-dimensional subspace of $\mathfrak{g}$.

\noindent
$(2)$ If $E$ has ordinary good reduction 
without complex multiplication,  then 
$\mathfrak{g}$ is the Borel subalgebra of 
$\mathrm{End}_{\mathbb{Q}_p}(V_p(E))$ corresponding to 
the kernel of the natural reduction map $V_p(E)\to V_p(\tilde E)$
and $\mathfrak{i}$ is a $2$-dimensional subspace of 
$\mathfrak{g}$.

\noindent
$(3)$ If $E$ has supersingular good reduction with 
formal complex multiplication, then
$\mathfrak{g}$ is a non-split Cartan subalgebra 
of $\mathrm{End}_{\mathbb{Q}_p}(V_p(E))$
and $\mathfrak{i}=\mathfrak{g}$.

\noindent
$(4)$ If $E$ has supersingular good reduction without 
formal complex multiplication, then
$\mathfrak{g}=\mathrm{End}_{\mathbb{Q}_p}(V_p(E))$
and $\mathfrak{i}=\mathfrak{g}$.

\noindent
$(5)$ If the $j$-invariant of $E$ has negative $p$-adic valuation, then 
$\mathfrak{g}$ coincides with $\mathfrak{n}_X$ for some $1$-dimensional 
subspace $X$ of $V_p(E)$
and $\mathfrak{i}=\mathfrak{g}$. 
Here $\mathfrak{n}_X$ is the subspace of $\mathrm{End}_{\mathbb{Q}_p}(V_p(E))$
generated by all $u\in \mathrm{End}_{\mathbb{Q}_p}(V_p(E))$ satisfying that
$u(V_p(E))\subset X$.

\end{proposition}

\noindent
Here,
for any elliptic curve $E$ over $K$ which has supersingular good reduction,
we say that $E$ has 
\textit{formal complex multiplication} over $K$
if an endomorphism ring of the $p$-divisible group $\mathscr{E}(p)$
over $\cO_K$ has rank $2$ 
as a $\mathbb{Z}_p$-module, 
where $\mathscr{E}(p)$ is 
the $p$-divisible group associated with the N\'eron model 
$\mathscr{E}$ of $E$ over $\cO_K$. 
We also say that 
$E$ has \textit{formal complex multiplication} if 
$E\times_K K'$ has formal complex multiplication 
defined over some algebraic extension $K'$ of $K$.
Then the quadratic field 
$\mathrm{End}_{\cO_{K'}}(\mathscr{E}(p))\otimes_{\mathbb{Z}_p} \mathbb{Q}_p$ 
is called the the field of formal complex multiplication.
We can take $K'$ for at most degree 2 extension of $K$.

Our first result in this subsection is :
%%%%%%%%%%%%%%%%%%%%%%%%%%%%%%%%%%%%%%%%%%%%%%%%%
%lemma%
%%%%%%%%%%%%%%%%%%%%%%%%%%%%%%%%%%%%%%%%%%%%%%%%%
\begin{proposition}
\label{Pro:5-10}Let $E_1$ and $E_2$ be elliptic curves over $K$.
Suppose that $E_1$ has 
potential supersingular good reduction over $K$.
Then $E_1(K_{E_2,p})[p^{\infty}]$ is finite if one of the 
following conditions is satisfied:

\noindent
$(1)$ The elliptic curve $E_2$ has 
potential ordinary good reduction.

\noindent
$(2)$ The elliptic curve $E_2$ has 
potential supersingular good reduction.
Furthermore, one of the following conditions is satisfied:

%\noindent
$(\mathrm{i})$ $E_1$ has formal complex multiplication but 
$E_2$ does not have. 

%\noindent
$(\mathrm{ii})$ $E_2$ has formal complex multiplication but 
$E_1$ does not have. 

\noindent
$(3)$ The $j$-invariant of $E_2$ has negative $p$-adic value. 
\end{proposition}

To prove this proposition, 
we need the following important lemma,
which is easy to prove, but we often make use of this lemma 
in this subsection. 
\begin{lemma}
\label{Lem:5-4}Let $E$ be an elliptic curve over $K$ which has 
supersingular good reduction over $K$.  
Let $L$ be a Galois extension of $K$.
Then the group $E(L)[p^{\infty}]$ is finite
if and only if $K_{E,p}$ is not contained in $L$.
\end{lemma}    

\begin{proof}   
Let us assume that $K_{E,p}$ is not contained in $L$.
Then we have $V_p(E)^{G_L}\subsetneq V_p(E)$.
By the assumption that $L$ is a Galois extension of $K$,
we see that $V_p(E)^{G_L}$ is a $G_K$-submodule of $V_p(E)$.
In addition, since $E$ has supersingular reduction,
$V_p(E)$ is an irreducible $G_K$-module 
(cf.\ \cite{Se89}, Thm. of A.2.2). 
Hence we see that $V_p(E)^{G_L}$ is vanished and 
this implies the fact that $E(L)[p^{\infty}]$ is finite.

The converse is obvious. 
\end{proof}     

\begin{proof}[Proof of Prop.\ \ref{Pro:5-10}]
First we note that the Lie algebras 
$\mathfrak{g}_i:=\mathrm{Lie}(\rho_{E_i,p}(G_K))$ and 
$\mathfrak{i}_i:=\mathrm{Lie}(\rho_{E_i,p}(I_K))$ are
subspaces of $\mathrm{End}(V_p(E_i))$ for $i=1,2$. 
By extending $K$ finitely, we may assume that 
$E_1$ has good reduction over $K$.
By Lem.\ \ref{Lem:5-4}, it suffices to show that
$K_{E_1,p}$ is not contained in $K_{E_2,p}$. 
Let us assume that $K_{E_1,p}$ is contained in $K_{E_2,p}$.
Then there is a natural surjection
$\mathrm{Gal}(K_{E_2,p}/K)\to \mathrm{Gal}(K_{E_1,p}/K)$
and hence we obtain the surjections of Lie algebras
$\mathfrak{g}_2\to \mathfrak{g}_1$
and $\mathfrak{i}_2\to \mathfrak{i}_1$.
Hence we see that 
$\mathrm{dim}_{\mathbb{Q}_p}\mathfrak{g}_2
\ge \mathrm{dim}_{\mathbb{Q}_p}\mathfrak{g}_1$
and $\mathrm{dim}_{\mathbb{Q}_p}\mathfrak{i}_2
\ge \mathrm{dim}_{\mathbb{Q}_p}\mathfrak{i}_1$.
If $\mathfrak{g}_1$ (resp.\ $\mathfrak{i}_1$) is abelian,
we can also obtain the inequality 
$\mathrm{dim}_{\mathbb{Q}_p}(\mathfrak{g}_2/[\mathfrak{g}_2,\mathfrak{g}_2])\ge
\mathrm{dim}_{\mathbb{Q}_p}\mathfrak{g}_1$
(resp.\ $\mathrm{dim}_{\mathbb{Q}_p}(\mathfrak{i}_2/
[\mathfrak{i}_2,\mathfrak{i}_2])\ge
\mathrm{dim}_{\mathbb{Q}_p}\mathfrak{i}_1$).
But
Prop.\ \ref{Pro:5-4} implies that 
at least one of the above inequalities is not satisfied.
\end{proof}

In the rest of this section, 
we consider about the case where
$E_2$ has supersingular good  reduction.
First we prove an elementary result of the algebraic number theory.

\begin{lemma}
\label{Lem:5-5}
Suppose $F_1$ and $F_2$ are quadratic subfields in $K$.
Denote by $\mathrm{Nr}_{K/{F_i}}$ the norm map 
$K^{\times}\to F_i^{\times}$ for $i=1,2$.
Then the followings are equivalent:

\noindent
$(a)$ $F_1\not =F_2$,

\noindent 
$(b)$ $\mathrm{ker}(\mathrm{Nr}_{K/F_2})\not \subset
\mathrm{ker}(\mathrm{Nr}_{K/F_1})$,

\noindent
$(c)$ $\mathrm{ker}(\mathrm{Nr}_{K/F_1})\not \subset
\mathrm{ker}(\mathrm{Nr}_{K/F_2})$.
\end{lemma}

\begin{proof}
Let $\alpha_i$ be a 
$p$-adic integer satisfying that 
$F_i=\mathbb{Q}_p(\alpha_i)$ 
and  
$\alpha_i^2\in \mathbb{Z}_p$ for $i=1,2$.
Assume that $F_1\not =F_2$.
Then 
\[
\mathrm{Nr}_{K/{F_2}}(x+y\alpha_1)=(x^2-y^2\alpha_1^2)^d,\ 
\mathrm{Nr}_{K/{F_1}}(x+y\alpha_1)=(x+y\alpha_1)^{2d},
\]
where $d$ is the extension degree of $K/{F_1F_2}$ and 
$x,y\in \mathbb{Z}_p$.
By the elementary field theory,
$\mathrm{ker}(\mathrm{Nr}_{K/F_2})|_{F_1^{\times}}$
has order at most $2d$.
Thus, to obtain the desired result,  
it is enough to show that 
the following claim:
\begin{claim}
The group $\mathrm{ker}(\mathrm{Nr}_{K/F_1})|_{F_1^{\times}}$ has an 
infinite order.
\end{claim}
\noindent
If we suppose $p\ge 3$, then 
the polynomial $T^2-y^2\alpha_1^2-1$ has roots in $\mathbb{Z}_p$
for all $y\in p\mathbb{Z}_p$ by Hensel's lemma and thus the claim is true.
Suppose $p=2$.
To prove the claim, it suffices to show that 
$x^2-y^2\alpha_1^2$ is equal to $1$
for infinitely many $x, y\in \mathbb{Z}_p$.
Take any integer $k\ge 0$ and suppose $y=4kx$.
In this case $x^2-y^2\alpha_1^2=(1-16k^2\alpha_1^2)x^2$ and 
we see  
$(1-16k^2\alpha_1^2)\in (\mathbb{Z}_p^{\times})^2$ 
since it can be checked that the polynomial 
$(1+2T)^2-(1-16k^2\alpha_1^2)\in \mathbb{Z}_p[T]$ has 
roots in $\mathbb{Z}_p$ by Hensel's lemma.
This completes the proof of the claim.
\end{proof}

\begin{proposition}
\label{Pro:5-6}
Let $E_1$ and $E_2$ be elliptic curves over $K$
which have supersingular good reduction 
with formal complex multiplication.
Let $F_i\subset K$ be the field of formal complex multiplication
for each $E_i$.
Then
$E_1(K_{E_2,p})[p^{\infty}]$ is finite if  
$F_1\not =F_2$.
\end{proposition}

\begin{proof}
By extending $K$ finitely,
we may assume that 
$E_1$ and $E_2$ have
formal complex multiplication over $K$.
Then the representation $\rho_{E_i,p}:G_K\to GL(V_p(E))$
has values in $O_{F_i}^{\times}$
and hence $\rho_{E_i,p}$ induces the representation
\[
\Psi_i:\cO_K^{\times}\simeq I(K^{\mathrm{ab}}/K)\to \cO_{F_i}^{\times}, 
\] 
where $I(K^{\mathrm{ab}}/K)$ is the inertia subgroup of 
$\mathrm{Gal}(K^{\mathrm{ab}}/K)$.
It is known that $\Psi_i(x)=\mathrm{Nr}_{K/F_i}(x^{-1})$ 
for any $x\in \cO_K^{\times}$ (see \cite{Se89}, A.2.2). 
Lam.\ \ref{Lem:5-4} implies that 
$E_1(K_{E_2,p})[p^{\infty}]$ is infinite 
only if 
$K^{\mathrm{ur}}(E_1(\bar K)[p^{\infty}])\subset 
K^{\mathrm{ur}}(E_2(\bar K)[p^{\infty}])$.
Because
$K^{\mathrm{ur}}(E_i(\bar K)[p^{\infty}])$
is the definition field of the representation 
$\rho_{E_i,p}|_{I_K}$
and 
$\mathrm{ker}(\mathrm{Nr}_{K/F_2})|_{\cO_K^{\times}}=
\mathrm{ker}(\mathrm{Nr}_{K/F_2})$ 
(resp.\ $\mathrm{ker}(\mathrm{Nr}_{K/F_1})|_{\cO_K^{\times}}=
\mathrm{ker}(\mathrm{Nr}_{K/F_1})$),
we see that 
$E_1(K_{E_2,p})[p^{\infty}]$ is infinite 
only if 
$\mathrm{ker}(\mathrm{Nr}_{K/F_2})\not \subset 
\mathrm{ker}(\mathrm{Nr}_{K/F_1})$.
Combining this fact and Lem.\ \ref{Lem:5-5},
we obtain the desired result.  
\end{proof}

For any elliptic curve $E$ over $K$,
we denote by $\hat E$ the formal group associated with $E$.
\begin{proposition}
\label{Pro:5-7}
Let $E_1$ and $E_2$ be elliptic curves over $K$
which have supersingular good reduction
without formal complex multiplication.

\noindent
$(1)$ If there is a non-trivial homomorphism of 
formal groups $\hat{E_2}\to \hat{E_1}$ over $\cO_K$,
then $E_1(K_{E_2,p})[p^{\infty}]$ is infinite.

\noindent
$(2)$ If $E_1(K_{E_2,p})[p^{\infty}]$ is infinite,
then there is a non-trivial homomorphism of 
formal groups $\hat{E_2}\to \hat{E_1}$ over $\cO_{K'}$
for some finite extension $K'$ of $K$.
\end{proposition}

\begin{proof}
(1) Since each $E_i$ has supersingular good reduction,
there is a non-trivial $G_K$-homomorphism $T_p(E_2)\to T_p(E_1)$
by a theorem of Tate (cf.\ \cite{Ta67},\ Cor.\ 1 of Thm.\ 4).
Hence Prop.\ \ref{Pro:2-0} 
implies that $E_1(K_{E_2,p})[p^{\infty}]$ is infinite.

\noindent
(2) Extending $K$ finitely, 
we may assume that 
$E_1$ and $E_2$ have 
formal complex multiplication over $K$.
By Lem.\ \ref{Lem:5-4}, we know that $K_{E_1,p}\subset K_{E_2,p}$.
Furthermore this is a finite field extension because 
the Lie groups $\mathrm{Gal}(K_{E_1,p}/K)$ 
and $\mathrm{Gal}(K_{E_2,p}/K)$ have same Lie algebras.  
Hence we may assume that $K_{E_1,p}$ coincides with $K_{E_2,p}$ by 
replacing $K$ with a finite extension.
Put $G:=\mathrm{Gal}(K_{E_1,p}/K)=\mathrm{Gal}(K_{E_2,p}/K)$
and $\mathfrak{g}:=\mathrm{Lie}(G)$.
Then the Lie algebra homomorphism 
$\mathrm{Lie}(\rho_{E_i,p}):\mathfrak{g}\to
\mathfrak{gl}_2(\mathbb{Q}_p)$
corresponding to 
$\rho_{E_i,p}:G\to 
GL(T_p(E_i))\simeq GL_2(\mathbb{Z}_p)$
is an isomorphism 
by our assumptions of $E_i$
for $i=1,2$.
\begin{claim}
$\mathrm{Lie}(\rho_{E_1,p})$ is conjugate with $\mathrm{Lie}(\rho_{E_2,p})$.  
\end{claim}

\noindent
If we finish the proof of this claim, 
we know that $\rho_{E_1,p}$ is conjugate with 
$\rho_{E_2,p}$ on $G_{K'}$ for some finite extension $K'$ of $K$
(this fact follows from \cite{DSMS}, Thm.\ 9.11).
Thus there is a non-trivial $G_{K'}$-homomorphism 
$T_p(E_2)\to T_p(E_1)$. 
By using the theorem of Tate again,
we see that there is a desired non-trivial homomorphism 
of formal groups $\hat{E_2}\to \hat{E_1}$ over $\cO_{K'}$.
 
Let us prove the above claim.
We denote by $Z(G)$ 
(resp.\ $Z_1$, $Z_2$)  
the center of $G$ 
(resp.\ $\mathrm{Im}(\rho_{E_1,p})$, $\mathrm{Im}(\rho_{E_2,p})$)
and put $\mathfrak{s}:=\mathrm{Lie}(G/{\mathrm{Gal}(K(\mu_{p^{\infty}})}/K))$.
We see that the homomorphisms $\mathrm{Lie}(\rho_{E_1,p})$ and 
$\mathrm{Lie}(\rho_{E_2,p})$ induce isomorphisms
$\mathrm{Lie}(Z(G))\simeq \mathrm{Lie}(Z_1)
=\mathbb{Q}_p \cdot I_2$ and 
$\mathrm{Lie}(Z(G))\simeq \mathrm{Lie}(Z_2)
=\mathbb{Q}_p\cdot I_2$, 
respectively.
Here  $I_2$ is the unit matrix of $2\times 2$.
We shall remark that 
$\mathrm{Lie}(Z(G))$ and $\mathfrak{s}$
generate $\mathfrak{g}$.
This can be checked from the facts that 
$\mathrm{dim}_{\mathbb{Q}_p}\mathrm{Lie}(Z(G))=1$,
$\mathrm{dim}_{\mathbb{Q}_p}\mathfrak{s}=3$ and 
\begin{align*}
\mathrm{Lie}({\rho_{E_1,p}})(\mathrm{Lie}(Z(G))\cap \mathfrak{s})&=
\mathrm{Lie}({\rho_{E_1,p}})(\mathrm{Lie}(Z(G)))\cap 
\mathrm{Lie}({\rho_{E_1,p}})(\mathfrak{s})\\
&=\mathbb{Q}_p\cdot I_2 \cap
\mathfrak{sl}_2(\mathbb{Q}_p)
=0.
\end{align*}
Since $\mathrm{det}(\rho_{E_1,p})$ and $\mathrm{det}(\rho_{E_2,p})$
coincide with the $p$-adic cyclotomic character, 
we have the equality that 
$\mathrm{Lie}(\rho_{E_1,p})= 
\mathrm{Lie}(\rho_{E_2,p})$ on $\mathrm{Lie}(Z(G))$.
On the other hand, up to inner automorphisms,
there exists the unique
Lie algebra injection from 
$\mathfrak{s}$
to $\mathfrak{gl}_2(\mathbb{Q}_p)$ because 
$\mathfrak{s}$
is isomorphic to $\mathfrak{sl}_2(\mathbb{Q}_p)$.
This fact is followed from the highest weight theory (cf.\ 
\cite{Se64},\ the remark of the end of Subsection 7.4) 
Thus $\mathrm{Lie}(\rho_{E_1,p})$ is conjugate with
$\mathrm{Lie}(\rho_{E_2,p})$ on $\mathfrak{s}$.
Thus we have proved the assertion of the claim. 
\end{proof}

\subsection{Multiplicative reduction case}

Finally we give some results on the finiteness of 
$A(K_{B,p})[p^{\infty}]$ for an elliptic curve
$A=E$ over $K$ with multiplicative reduction 
and an abelian variety $B$ over $K$.
We always suppose this assumption to the end 
of this subsection.   
Since $E$ has multiplicative reduction,
we may choose a suitable basis of $T_p(E)$
such that the corresponding matrix of $\rho_{E,p}$ has the form
\[
\left(\begin{matrix}
  \e\chi^{-1} & u \\  0 & \chi
  \end{matrix}\right), 
\]
where
$\chi:G_K\to \mathbb{Z}_p^{\times}$ is a character 
which has an image
of order at most $2$ (cf.\ \cite{Co97}, Thm.\ 1.1).
Moreover $\chi$ is trivial if and only if 
$E$ has split multiplicative reduction over $K$.
We identify $T_p(E)$ as $\mathbb{Z}_p^{\oplus 2}$
with respect to the above basis.

\begin{proposition}
\label{Pro:5-2}
Let $E$ be an elliptic curve which has 
split multiplicative reduction over $K$.
Let $K(\mu_{p^{\infty}})$ be an algebraic extension of $K$  
adjoining all $p$-power roots of unity.
Then 
$E(K(\mu_{p^{\infty}}))[p^{\infty}]$ is infinite.
In particular, $E(K_{B,p})[p^{\infty}]$ is infinite
for any abelian varieties $B$ over $K$.
\end{proposition}

\begin{proof}
By the theory of Tate curve, we have 
a $G_K$-isomorphism
\[
E(\bar K)[p^{\infty}]\simeq
\bar K^{\times}/{q^{\mathbb{Z}}}  
\]   
for some element $q$ in the maximal ideal of $\cO_K$ and 
hence we have the desired result. 
\end{proof}

Next we shall consider the case where
$E$ has non-split multiplicative reduction, equivalently,
$\chi\not =1$.
%%%%%%%%%%%%%%%%%%%%%%%%%%%%%%%%%%%%%%%%%%%%%%%%%%%%%%
%lemma%
%%%%%%%%%%%%%%%%%%%%%%%%%%%%%%%%%%%%%%%%%%%%%%%%%%%%%
\begin{lemma}
\label{Lem:5-3}
Suppose that $E$ has non-split multiplicative reduction over $K$.
Let $L$ be an algebraic extension of $K$ which 
contains all $p$-power roots of unity.
Then $E(L)[p^{\infty}]$ is finite 
if and only if $\chi$ is not trivial on $G_L$.
\end{lemma}
\begin{proof}
Since the field $L$ contains all $p$-power roots of unity,
the matrix
$\left(\begin{smallmatrix}
  x \\ y
\end{smallmatrix}\right)\in T_p(E)$
is contained in $T_p(E)^{G_L}$ if and only if 
\[
\chi(\sigma)^{-1}x+u(\sigma)y=x,\quad \chi(\sigma)y=y 
\]
for all $\sigma\in G_L$.
The assertion immediately follows from this fact.  
\end{proof}

Applying this Lemma,
we can show the following finiteness results 
without difficulty.
%%%%%%%%%%%%%%%%%%%%%%%%%%%%%%%%%%%%%%%%%%%
%prop%
%%%%%%%%%%%%%%%%%%%%%%%%%%%%%%%%%%%%%%%%%%%%
\begin{proposition}
\label{Pro:5-3}
Let $E$ be an elliptic curve 
which has non-split multiplicative reduction over $K$.
We suppose that $p\not =2$. 

\noindent
$(1)$ If $L$ is a prime-to-$2$ algebraic extension
of $K$ 
which contains all $p$-power roots of unity,
then $E(L)[p^{\infty}]$ is finite.

\noindent
$(2)$ Let $B$ be an abelian variety such that 
$B(\bar K)[p]$ is rational over $K$. 
Then 
$E(K_{B,p})[p^{\infty}]$ is finite.
\end{proposition}
\begin{proof}
Since the assertion $(2)$ is a special case of $(1)$,
it is enough to show the assertion $(1)$. 
Let $K_{\chi}$ be the definition field of the character $\chi$. 
By the assumption that $L$ is a  profinite prime-to-$2$ algebraic extension
of $K$, we know that $L\cap K_{\chi}=K$, in particular
$K_{\chi}$ is not contained in $L$.
Consequently 
it follows that $E(L)[p^{\infty}]$ is finite
by Lem.\ \ref{Lem:5-3}.
\end{proof}

%\bibliographystyle{amsalpha}
%\bibliographystyle{amsplain}
%\bibliography{ronbunver2}
\end{document}